\documentclass[12pt, a4paper]{amsart}
\usepackage{amssymb}
\usepackage[T1]{fontenc}

\normalbaselines
\theoremstyle{plain}
\newtheorem{theorem}{Theorem}
\newtheorem{corollary}[theorem]{Corollary}
\newtheorem{lemma}[theorem]{Lemma}

\theoremstyle{definition}
\newtheorem{definition}[theorem]{Definition}

\newcommand{\suchthat}{:} 
\newcommand{\group}[1]{{\mathrm{#1}}}
\newcommand{\Moeb}{\group{M\ddot{o}b}}

\newcommand{\R}{{\mathbb{R}}}
\newcommand{\N}{{\mathbb{N}}}

\newcommand{\Mak}{{\mathcal{M}}}
\newcommand{\Hyp}{{\mathcal{H}}}

\begin{document}

\title[Rigidity of causal maps]{Rigidity of causal maps}

\author[N.~B.~Andersen]{Nils Byrial Andersen}
\address{\mbox{Department of Mathematics}\\ \mbox{Aarhus University}\\\newline
\mbox{Ny Munkegade 118}\\ \mbox{DK-8000 Aarhus C}\\ Denmark}
\email{byrial@imf.au.dk}

\thanks{The authors were supported by a research grant from the Australian Research Council (ARC)}

\author[M.~G.~Cowling]{Michael G.~Cowling}
\address{\mbox{School of Mathematics and Statistics}\\ \mbox{University of New South Wales}\\\newline
\mbox{Sydney NSW 2052}\\ Australia }
\email{m.cowling@unsw.edu.au}

\date{17/07/2013}

\begin{abstract}
We show that order-invariant injective maps on the noncompactly causal symmetric space
$\group{SO} _0 (1,n)/\group{SO}_0 (1,n-1)$ belong to $\group{O}(1,n)^+$.
\end{abstract}

\maketitle

\section{Introduction}

Let $G$ be $\group{SO} _0 (1,n)$, let $H$ be $\group{SO}_0 (1,n-1)$, 
which we identify with a subgroup of $G$ in the usual way, and let $\Mak$ be the quotient space $G/H$.
Then $\Mak$ is a noncompactly  causal semisimple symmetric space, 
that is, there exists a $G$-invariant global partial order determined infinitesimally by
an $H$-invariant cone $C_{eH}$ in the tangent space $T _{eH} G/H$ at the origin $eH$.
More precisely, $C_{eH}$ is the Lorentz cone, and the order is defined by $yH\ge xH$ if and only if
$yx^{-1}\in \exp(C_{eH})$, where $x,y\in G$. If we identify $\Mak$ with $\Hyp _n$,
the hyperboloid with one sheet in $\R ^{n+1}$:
\begin{equation*}
\Hyp _n:= \{ x = (x_0,x_1\dots,x_{n}) \in \R^{n+1} \suchthat
-x_0 ^2 +x_1^2 +\dots +x_n^2 =1\},
\end{equation*}
then the partial order is given by
\begin{equation*}
y\ge x \iff  (y_0 \ge x_0 \quad {\text{and}} \quad  -x_0y_0 +
x_1y_1\dots +x_n y_n\ge 1).
\end{equation*}
The space $\Mak$ may also be realised as a Makarevi\v c space,
that is, as an open symmetric orbit in the conformal
compactification $V^c$ of a semisimple Euclidean Jordan algebra
$V$. Then $\exp V$, which will be written $\Omega$, is a  symmetric convex cone that defines the
(flat) invariant causal structure on $\Mak$.

In this paper we characterise the order-invariant (or, cone-preserving)
injective maps $f$ on $\Mak$, without making any continuity or
differentiability assumptions, nor assuming that $f$ is surjective.
Our strategy is to study the
behaviour of $f$ at the future boundary of $\Mak$ (as $x_0 \to
\infty$), which in the Makarevi\v c picture can be identified with
$\R^{n-1}\cup \{\infty\}$. We define an (unique) extension of $f$
at $\infty$, and show that this is a M\"obius map, which in turn
implies that $f$ is causal, and hence birational and smooth.

Geometrically, noncompactly causal symmetric spaces may be viewed
as generalisations of the space-times arising in general
relativity. In this physical set-up, Alexandrov \cite{A1}, Zeeman \cite{Z}
and others have showed that bijective maps on Minkowski space
preserving causality (or preserving timelike or null cones) are
given by Lorentz transformations, translations and dilations.
There are generalisations to other space-times, and Alexandrov
\cite{A2} and Lester \cite{L} have also considered conformal
spaces.

\section{The imaginary Lobachevski space $\Mak \simeq \group{SO} _0 (1,n)/\group{SO}_0 (1,n-1)$}

In the following two sections, we summarise the notions and results that we need
about causal symmetric spaces and Makarevi\v c spaces.
For details, we refer to \cite{B2, B3, F, FHO, HO}.
In particular, we use the examples in \cite[\S8]{F} and \cite[\S10]{FHO}.

Let $V$  be the Euclidean space $\R ^{n}$, equipped with the usual Euclidean scalar product, 
and write $z\in \R ^{n}$ in coordinates as $(z_0,z_1,\dots, z_{n-1})$.
Define the product $z = x y$ by
\begin{equation*}
z_0 = x_0 y_0 +\dots + x_{n-1} y _{n-1}
\quad \text{and}\quad
z_j = x_0 y_j + x_j y_0
\quad\text{if $ j=1,\dots, n-1$}.
\end{equation*}
With this additional structure, $V$ is the Euclidean Jordan algebra $\R^{1,n-1}$, 
and its neutral element $e$ is $(1,0,\dots,0)$.
Let $\Delta$ denote the quadratic Lorentz form:
\begin{equation*}
\Delta (x) = x_0 ^2 - x_1 ^2 -\dots - x_{n-1} ^2.
\end{equation*}
The associated open symmetric cone $\Omega$ is the Lorentz cone:
\begin{equation*}
\Omega = \{x^2\}^o= \{x\in V \suchthat \Delta (x) >0,\,   x_0 >0 \}.
\end{equation*}

Let $j\colon x \mapsto x^{-1}$ be the inversion map on $V$, let $N_+\simeq V$ 
be the translation group $\{\tau_v \suchthat v\in V\}$, where  $\tau _v( x):= x+v$, 
and let $\group{Str}(V)$ denote the structure group:
\[
\group{Str}(V) :=  \{g\in \group{GL} (V) \suchthat j \circ g\circ j\in \group{GL} (V)\} .
\]
Note that $\group{Str}(V)= G(\Omega)\times\{\pm I\}$, where
\[
G(\Omega)=  \{ g\in \group{GL}(V)\suchthat g(\Omega) \subseteq \Omega \}= \R _{+} \times \group{O}(1,n-1)^+.
\]
The conformal group $\group{Co}(V)$ of $V$ is the group of rational maps generated by
$N_+$, $\group{Str}(V)$ and the inversion map $j$, 
and the causal group $\group{Co}(G(\Omega))$ is the subgroup generated by $N_+$, $G(\Omega)$ and $j$.

Define $N_- := j\circ N_+ \circ j$, then
$\group{Str}(V) \ltimes N_-$ is a parabolic subgroup of $\group{Co}(V)$. The
conformal compactification $V^c$ of $V$ is defined by $V^c
:=  \group{Co}(V)/\group{Str}(V) \ltimes N_-$. We equip $V^c$ with the flat $\group{Co}
(G(\Omega))$-invariant causal structure defined by
$-\overline{\Omega} \subseteq T _x V^c = V$.

Define  the involution $\alpha$ on $V$ by
\begin{equation*}
\alpha (x_0, x_1, \dots, x_{n-1}) = (x_0, -x_1, \dots,- x_{n-1});
\end{equation*}
the eigenspaces of $\alpha$ are $V_+$ and $V_-$, where $V_+= \{ (x_0, 0, \dots, 0)\suchthat x_0 \in \R\}$ and
\[
V_- = \{
(0, x_1, \dots, x_{n-1})\suchthat (x_1, \dots x_{n-1}) \in
\R^{n-1}\}\cong \R^{n-1}.
\]
Further, let
\[
\Omega _+ =\Omega \cap V_+= \{(x_0, 0, \dots, 0)\suchthat x_0 \in \R_+\},
\]
and define $x_+ := (x +\alpha (x)) /2$ and $x_- := (x -\alpha (x)) /2$, for all $x\in V$.

Let $G$ be the group
\begin{equation*}
G:=  \group{Co}(V) _0 ^{(-\alpha )}= \{g\in \group{Co}(V) \suchthat (-\alpha )\circ g\circ (-\alpha) = g \} _0
\end{equation*}
(the subscript $0$ indicates the identity component).
Then $G$ is generated by translations by elements of $V_-$, dilations $x \mapsto rx$, 
where $r >0$, the orthogonal maps in $\group{O}(n-1)$ and $j$.
The orbit $\Mak = G e$ is open in $V^c$, and $\Mak \simeq G/H$, where
$H := \{ g\in G \suchthat g  e = e\}$. The involution $\sigma\colon g \mapsto j\circ g\circ j$
makes $\Mak = G/H$ into a symmetric space of Makarevi\v c type, with flat causal structure
inherited from $V^c$.
Further, $G \simeq \group{SO} _0 (1,n)$ and $H \simeq \group{SO} _0 (1,n-1)$.

Let $\Hyp _n$ be the one-sheeted hyperboloid 
$\{y\in \R^{n+1} \suchthat -y_0 ^2 +y_1 ^2 +\dots + y_{n} ^2 =1\}$ in $\R^{n+1}$.
The connected component of $e$ of the subset $\Mak\cap V$ is given by
\begin{equation*}
(\Mak \cap V ) _e = \{ x\in V \suchthat x +\alpha (x) \in \Omega\}_e = V_- +\Omega _+,
\end{equation*}
and the map from $\Mak$ to  $\Hyp _n$  determined by
\begin{equation}\label{1}
y_0 = \frac {1-\Delta (x)}{2x_0},\quad  y_1 = \frac{x_1}{x_0}, \quad \dots,\quad
y_{n-1} = \frac{x_{n-1}}{x_0}, \quad y_n = \frac {1+\Delta (x)}{2x_0}
\end{equation}
for all $x\in (\Mak \cap V)_e$ is a causal isomorphism, with inverse given by
\begin{equation}\label{2}
x_0 = \frac 1{y_0 +y_{n}},\quad x_1 = \frac {y_1}{y_0 +y_{n}},\quad\dots,\quad
 x_{n-1} = \frac {y_{n-1}}{y_0 +y_{n}}.
\end{equation}

\section{The future}

A continuous, piecewise differentiable curve $\gamma\colon  [a,b] \to \Mak$ is
causal, if the (right) derivative $\gamma '(t)$
belongs to $\gamma(t)-{\overline{\Omega}}$ for all $t \in ( a,b )$.
We define the partial order $\le$ on $\Mak$ by $x\le y$ if and only if there
exists a causal curve from $x$ to $y$.
For each $z\in \Mak$, we say that $x$ belongs to the future of $z$ if
$x\ge z$, and we write $\Mak _z^+ := \{ x\in \Mak \suchthat x\ge z\}$. We
now show that with our choice of causal structure, the future of
$z\in (\Mak \cap V ) _e$ does not contain \lq points at $\infty$\rq.

\begin{lemma}
Let $z\in (\Mak \cap V ) _e$. The \lq future of $z$\rq\ is given by
\begin{equation*}
\Mak _z^+ = \{ x\in \Mak \suchthat x\ge z\} = (z-{\overline{\Omega}}) \cap
(x_+ \in \Omega _+).
\end{equation*}
This is a bounded set.
\end{lemma}
\begin{proof}
Let $x \in (z-{\overline{\Omega}}) \cap (x_+ \in \Omega _+)$, and
define $\gamma \colon  [0,1] \to V$ by $\gamma (t) = tx
+(1-t) z$. Then $(\gamma (t)) _+ = (\gamma (t)  +\alpha (\gamma
(t)))/2 \in \Omega$ by convexity, and $\gamma' (t) = x-z \in -
\overline{\Omega}$. So $\gamma$ is a nontrivial causal curve in
$\Mak$ from $z$ to $x$, and $x$ belongs to the future of $z$.

Conversely, let $\gamma \colon   [a,b] \to \Mak$ be a causal curve, with
$\gamma (a) = z$, that contains \lq points at $\infty$\rq, and define
\begin{equation*}
\kappa := \inf \{ t\in [a,b] \suchthat \gamma (t) \not\in \Mak \cap V\}.
\end{equation*}
The curve $\gamma \colon  [ a,\kappa) \to V$ is causal
(with causal structure given by $-\overline{\Omega}$) and
$\gamma  \subseteq z - \overline{\Omega}$.
Now $\gamma $ is
contained in $\{ x\in V\suchthat x +\alpha (x) \in \Omega \}$, and so
$\gamma \subseteq (z-{\overline{\Omega}}) \cap (x_+ \in
\Omega _+) $, which is a bounded set, and we have a contradiction.
\end{proof}

The case in which $z=e$ was proved in \cite[Proposition 5.2]{BS}. 
We define the \lq future boundary\rq\ of $\Mak _z^+$ in $V^c$  by $\partial _\infty
\Mak _z^+ := \overline{\Mak _z^+}\setminus\Mak _z^+$.
\begin{corollary}
Let $z\in (\Mak \cap V ) _e$. The \emph{future boundary} of $\Mak _z^+$ is given by
\begin{equation*}
\partial _\infty \Mak _z^+  = \overline{\Mak _z^+}\setminus\Mak _z^+= (z-{\overline{\Omega}})
\cap (x_+\in \overline{\Omega} _+\setminus\Omega _+).
\end{equation*}
\end{corollary}
\begin{proof}
Suppose that $x\in (z-{\overline{\Omega}}) \cap (x\in \overline{\Omega}
_+\setminus\Omega _+)$; define the curve $\gamma \colon[0,1] \to V^c$
by $\gamma (t) = tx +(1-t) z$. Then $\gamma (t) \in\Mak _z^+ $
for all $t\in [0,1),$ whence $x\in\overline{\Mak _z^+}$, but $x \notin \Mak _z^+$.

Conversely, $\partial _\infty \Mak _z^+  \subseteq
(z-{\overline{\Omega}}) \cap (x\in \overline{\Omega _+})$.
\end{proof}

In coordinates, we may write
\begin{align*}
\Mak _z ^+ &=\{ x\in V_- +\Omega _+\suchthat \Delta (z-x) \ge 0,\,z_0 - x_0 \ge 0,\, x_0 >0\}\\
& = \{ (z_0-x_0, z_1 - x_1, \dots, z_{n-1}-x_{n-1}) \suchthat x_1 ^2
+\dots + x_{n-1} ^2\le x_0 ^2,\, 0< z_0-x_0 \le z_0\},
\end{align*}
for all $z\in (\Mak \cap V ) _e$, and the future boundary $\partial _\infty \Mak _z ^+$ of $\Mak _z ^+$ is
\begin{equation*}
\partial _\infty \Mak _z ^+ =
\{ (0, z_1- x_1, \dots, z_{n-1} - x_{n-1}) \suchthat x_1 ^2 +\dots +
x_{n-1} ^2\le z_0 ^2\},
\end{equation*}
the ball around $(z_1,\dots z_{n-1})$ with radius $z_0$.
Here and later in this paper, a \emph{ball} means a closed ball of finite but strictly positive radius.
We
note that any ball in $\R^{n-1}$ can be represented as
the future $\partial _\infty \Mak _z ^+ $ of an element $z\in (\Mak
\cap V) _e$.

\section{M\"{o}bius maps in ${\R}^{n-1}\cup \{\infty \}$}

We denote a point $x$ in ${\R}^{n-1}$ by $(x_{1},\ldots,x_{n-1})$, the 
Euclidean distance between $x$ and $y$ in ${\R}^{n-1}$
by $|x-y|$,  and the one-point compactification ${\R}^{n-1}\cup \{\infty \}$ by ${\R}^{n-1}_{\infty }$.

The \emph{hyperplane}  in ${\R}^{n-1}_{\infty }$ determined by distinct points
$a,b\in{\R}^{n-1}$ is the set of points $x\in{\R}^{n-1}_{\infty }$ (including $\infty$)
such that $|x-a|=|x-b|$;
the reflection in the hyperplane is the isometry of $\R^{n-1}$ that fixes the points in the 
hyperplane and exchanges $a$ and $b$; it is extended to ${\R}^{n-1}_{\infty }$ in the obvious way.

The \emph{hypersphere} (or Euclidean sphere) determined by a point $a$ in ${\R}^{n-1}$ 
and a positive number $r$ is the set of points $x\in{\R}^{n-1}$ satisfying $|x-a|=r$;  
the reflection (or inversion) in this sphere is the map $x \mapsto a + (r / |x-a|)^{2} (x-a)$, 
defined on ${\R}^{n-1}_{\infty }$ in the obvious way.
In particular, let $j_- (x) := |x|^{-2} x$ be the reflection in the unit hypersphere.

We define a \emph{sphere} to be either a hyperplane or a hypersphere.
A \emph{M\"{o}bius map} $\phi$ on
${\R}^{n-1}_{\infty }$ is a composition of a finite number
of reflections in spheres, and the group of
M\"{o}bius maps on ${\R}^{n-1}_{\infty }$ is
denoted by $\Moeb ({\R}^{n-1}_{\infty })$. We refer to
\cite{Be} for a treatment of the general theory of M\"{o}bius maps.

A map $\phi$  of $\R ^{n-1}$ is a similarity if there exists a positive number $r$ such that
\begin{equation*}
|\phi(y) - \phi (x)| = r |y-x|
\end{equation*}
for all $x,y \in \R^{n-1}$. Given such a map $\phi$, there exists an orthogonal matrix
$A$ such that $\phi (x) = r Ax +x_0$, where $x_0 \in \R^{n-1}$.
Similarities are M\"obius maps, and the M\"obius maps are generated by similarities and the reflection $j_-$.

Every M\"obius map $\phi$ has the property that either $\phi^{-1}(\R^{n-1}) = \R^{n-1}$,
in which case $\phi$ is a similarity,
or there is a point $x_0$ in $\R^{n-1}$ such that $\phi(x_0) = \infty$.
In the first case, $\phi$ maps spheres to spheres and the interiors of spheres 
(the bounded connected component of the complement of the sphere) to the interiors of spheres.
In the second case, $\phi$ maps spheres and their interiors to spheres and their interiors, 
as long as $x_0$ is not an element of the sphere or its interior.

M\"{o}bius maps have a unique continuation property, namely,
if $f$ is a map of a domain $D$ in ${\R}^{n-1}_{\infty }$
into ${\R}^{n-1}_{\infty }$ and for each point $a$ in
$D$ there is a ball $B$ containing $a$ such that
the restriction of $f$ to $B$ coincides with the restriction to $B$ of an element of $\Moeb ({\R}^{n-1}_{\infty })$, 
then $f$ is the restriction to $D$ of an element of $\Moeb ({\R}^{n-1}_{\infty })$.
This follows immediately from the fact that M\"{o}bius maps are completely determined 
by their restrictions to arbitrarily small nonempty open sets.

The next result \cite[Theorem 2.1]{H} generalises a theorem of Carath{\'e}odory
\cite{C},  which treats the case in which $n=2$.

\begin{theorem}\label{CC}
Let $D$ be a domain in $\R^{n-1}$, where $n\ge 3$.
Let $f\colon D\to\R^{n-1}$ be an injective map such that $f(H)$ is a
hypersphere whenever $H$ is a hypersphere contained in $D$ whose interior is contained in $D$.
Then $f$ is the restriction of a M{\"o}bius map.
\end{theorem}

Here, the \lq interior\rq\  means the bounded component of the complement of the hypersphere.
Our next result  is a corollary.

\begin{corollary}\label{btb}
Let $D$ be a domain in $\R^{n-1}$, where $n\ge 3$.
Let $f\colon D\to\R^{n-1}$ be an injective map such that $f(B)$ is a
ball whenever $B$ is a ball contained in $D$. Then $f$ is
the restriction to $D$ of a M{\"o}bius map.
\end{corollary}
\begin{proof}
Let $B$ be a ball in the domain $D$.
It is easy to verify that a point $x$ in $B$ is an interior point if and only if there are two balls $B_1$ and $B_2$
contained in $B$ such that $B_1 \cap B_2 = \{x\}$, and is a boundary point if and only if there is a ball $B_3$ contained in $D$ such that
$B \cap B_3 = \{x\}$.

Now let $B$ be a ball contained in $D$.
Then $f(B)$ is a ball in $f(D)$; it must be infinite and cannot be all of $\R^{n-1}$ since $f$ is injective.
From the characterisation above, $f$ maps interior points of $B$ to
interior points of $f(B)$ and boundary points to boundary points.
Since $f(B)$ is a ball, the restriction of $f$ to the interior of $B$ maps onto the interior of $f(B)$, and the restriction of $f$ to the boundary of $B$ maps onto the boundary of $f(B)$.
So $f$ satisfies the hypothesis of Theorem \ref{CC}, and hence is a M\"obius map.
\end{proof}

\section{The boundary map}

\begin{definition}
An  map $f\colon \Mak \to \Mak$ is said to be \emph{conal} if it is injective and
maps $\Mak _z^+$ onto $\Mak _{f(z)}^+$ for all $z\in \Mak$.
\end{definition}

In other words, a conal map is an injective order-preserving map.
We note that the composition of two conal maps again is a conal
map.

\begin{definition}
A sequence $\{z_i\}_{i\in \N}\subseteq \Mak$ is said to be increasing if
$z_{i+1}\ge z_i$ for all $i\in \N$.
\end{definition}

\begin{lemma}\label{seq}
Suppose that $f$ is conal,  $z\in (\Mak \cap V)_e$, and $f(z)\in (\Mak
\cap V)_e$. Let $\{z_i\}_{i\in \N}$ be an
increasing sequence in $\Mak _z ^+$  that converges (in the Euclidean sense) to $z_-$ in $V_-$ as $i\to
\infty$. Then the
(increasing) sequence $\{f(z_i)\}_{i\in \N}$ in $\Mak _{f(z)}
^+$ also converges to an element of $V_-$.
\end{lemma}
\begin{proof}
Take $\xi_-$ in $\bigcap _{i=1} ^\infty \partial _\infty \Mak _{f(z_i)} ^+$, which is
a nonempty intersection of balls in $V_-$.
We claim that $f(z_i)\to \xi_ -$ as $i\to \infty$.

The radius of the ball $\partial _\infty \Mak _{f(z_i)} ^+$ and the
height of the cone $\Mak _{f(z_i)} ^+$ is equal to $(f(z_i))_0$ (the
$0$th coordinate of $f(z_i)$). This gives an upper bound on the
Euclidean distance between $f(z_i)$ and $\xi _-$, namely, $|f(z_i) - \xi
_-| \le \sqrt{2} (f(z_i))_0$. Assume that $(f(z_i))_0$ does not
converge to zero. Then the intersection of cones $\bigcap _{i=1}
^\infty \Mak _{f(z_i)} ^+\subseteq (\Mak \cap V)_e$ is nonempty. Let
$\xi \in\bigcap _{i=1} ^\infty \Mak _{f(z_i)} ^+$. Then $z_i \le
f^{-1} (\xi)$ for all $i \in \N$, in particular $(z_i)_0 \ge
(f^{-1} (\xi))_0 >0$ for all $i \in \N$, and $z_i$ does not
converge to an element in $V_-$, which is a contradiction.
\end{proof}

\begin{definition}
Let $f$ be conal and assume that $z,f(z)\in (\Mak \cap V)_e$. We
define a local boundary map $f_z^-$ on $\partial _\infty \Mak _z ^+$
by
\begin{equation*}
f_z^-(z_-) :=  \lim _{z_i \to z_-}  f (z_i)
\end{equation*}
for all $z_- \in \partial_\infty \Mak _z ^+$, where $\{z_i\}_{i\in \N}\subseteq \Mak _z ^+$ is an arbitrary increasing
sequence such that $z_i\to z_-$ as $i\to \infty$.
\end{definition}

It is clear that $f_{z} ^-$ is well-defined
and injective on $\partial _\infty \Mak _z ^+$. We also note that
$f_{z} ^-$ maps the ball $\partial _\infty \Mak _{z'} ^+$
onto the ball $\partial _\infty \Mak _{f (z')} ^+$ for any
$z' \in \Mak _z^+$.

\begin{lemma}
Let $f$ be conal and assume that $z,f(z)\in (\Mak \cap V)_e$. Then
$f_{z}^-$ is the restriction of a M\"obius map on the interior
$\left (\partial _\infty \Mak _z ^+\right )^o$ of $\partial _\infty
\Mak _z ^+$.
\end{lemma}
\begin{proof}
The result follows from Corollary \ref{btb}, since $f_{z}^-$ maps
the ball $B =\partial _\infty \Mak _{z'} ^+$, for $z' \in
\left (\partial _\infty \Mak _z ^+\right )^o$, onto the ball
$f(B) = \partial _\infty \Mak _{f (z')} ^+$.
\end{proof}

Let $\varphi \in \group{Co}(G(\Omega)) ^{(-\alpha )}$. Then $\varphi$ is
a birational function defined on $V^c$ and $\varphi ^- =\varphi
_{V_-}$ by continuity. Let $f$ be conal and $z\in (\Mak \cap V)_e$.
By transitivity there exists a causal map
$\varphi\in\group{Co}(G(\Omega)) ^{(-\alpha )}$ such that $\varphi  \circ
f (z) \in(\Mak \cap V)_e$. The map $\varphi  \circ f$ is conal and
maps $\partial _\infty \Mak _z ^+$ onto $\partial _\infty \Mak
_{\varphi \circ f(z)} ^+$.

Now let $z_1,z_2\in (\Mak \cap V)_e$, and $\varphi_1, \varphi
_2\in\group{Co}(G(\Omega)) ^{(-\alpha )}$. Assume that $(\varphi _1
\circ f)^-_{z_1}$ and $(\varphi _2 \circ f) ^-_{z_1}$ are both
defined in a small neighbourhood around $z_-\in \partial _\infty
\Mak _{z_1} ^+\cap \partial _\infty \Mak _{z_2} ^+$.
The continuity of the causal maps implies that
\begin{align*}
(\varphi _1 \circ f)^- (z_-)
&= \lim _{z_i \to z_-}\varphi _1 \circ f(z_i) \\
&=\lim _{z_i \to z_-} \varphi _1 \circ \varphi _2 ^{-1}\circ \varphi _2 \circ f(z_i)\\
&=\varphi _1 \circ \varphi _2 ^{-1}\circ\lim _{z_i \to z_-}\varphi _2 \circ f(z_i) \\
&=\varphi _1 \circ \varphi _2 ^{-1}\circ(\varphi _2 \circ f)^- (z_-).
\end{align*}
Now we define the boundary map globally.

\begin{definition}
Let $f$ be conal.
We define a (global) boundary map $f^-$ on $V_-$ by
\begin{equation*}
f^- (z_-) := \varphi ^{-1} \circ ( \varphi  \circ f  )^-
(z_-),\qquad (z_-\in \partial _\infty \Mak _z ^+ \subset V_-),
\end{equation*}
where $\varphi$ is any map in $\group{Co}(G(\Omega)) ^{(-\alpha )}$ such
that $\varphi \circ f (z) \in(\Mak \cap V)_e$.
\end{definition}

\begin{lemma}
Let $f$ be conal. Then the boundary map $f^-$ is a
M\"obius map.
\end{lemma}
\begin{proof}
Let $B$ be a ball in $V_- = \R^{n-1}$. Then $B =\partial
_\infty \Mak _z ^+$, for some $z\in (\Mak\cap V)_e$. The restriction of
$f^-$ to $\left (\partial _\infty \Mak _z ^+\right )^o$ coincides with the
restriction of a M\"obius map. The unique continuation
property of M\"obius maps gives the result.
\end{proof}

It is worth noting that we have only used the behaviour of the conal
functions on the subset $(\Mak \cap V)_e =V_- +\Omega _+$, so our
result holds for injective order-preserving maps from $(\Mak \cap
V)_e$ into $\Mak$. Actually, the boundary map $f^-$ only
depends on the behaviour of $f$ (arbitrarily) close to the
boundary $V_-$. Let $\{U_i\}_{i\in I} =\{\partial _\infty \Mak
_{z_i} ^+\}_{i\in I}$ be a covering of $V_-$ consisting of closed
balls of radius less than some $\varepsilon >0$. Then $f$ is
determined by its values on $\bigcup_{i\in I}\Mak _{z_i} ^+$, which
lie inside an $\varepsilon$-band around $V_-$.

\section{Order invariant injective maps are causal}


Since different conal maps define different boundary maps, that
is, the transformation $f\mapsto f^-$ is injective on the set of conal
maps, we have the following rigidity theorem.

\begin{theorem}
Let $f$ be conal. Then $f$ extends to an element in
$\group{Co}(G(\Omega)) ^{(-\alpha )}$.
In particular,  $f$ extends to a birational function on $V^c$.
\end{theorem}

Finally we transfer our results to the hyperboloid with one sheet
in $\R^{n+1}$, using the isomorphism given by (\ref{1}) and
(\ref{2}). With notation as in \cite[\S10]{FHO}, it is easily seen
that $\R^+$ (dilations) corresponds to the hyperbolic
multiplication by the Abelian group $A$, the nilpotent groups
$N_+^\alpha$ and $N_-^\alpha$ correspond to the nilpotent groups
$N$ and $\overline{N}$ respectively and the (connected) component
of $\group{O}(n-1)$ corresponds to the compact group $K\cap H$,\ where $K
= \group{SO}(n)$ is the maximal compact subgroup of $G$.

\begin{corollary}
Let $f$ be an order-preserving injective map on the noncompactly
causal symmetric space $\group{SO} _0 (1,n)/\group{SO}_0 (1,n-1)$. Then $f$
extends to an element in $\group{O} (1,n)^+$.
\end{corollary}


\begin{thebibliography}{99}
\def\journal{\textit}
\def\bookname{\textit}
\def\article#1{#1,}
\def\volno#1{\textbf{#1}}
\bibitem{A1}
A. D. Alexandrov,
\article{A contribution to chronogeometry}
\journal{Canad. J. Math.} \volno{19} (1967), 1119--1128.

\bibitem{A2}
A. D. Alexandrov,
\article{On the foundations of relativity theory}
\journal{Vestnik Leningrad. Univ. Math.} \volno{19} (1976), 5--28.

\bibitem{Be}
A. F. Beardon,
\bookname{The Geometry of Discrete Groups}
Graduate Texts in Mathematics, 91,
Springer-Verlag, New York--Heidelberg--Berlin, 1983.

\bibitem{BS}
S. Ben Sa\"{\i}d,
\article{Weighted Bergman spaces on bounded symmetric domains}
\journal{Pacific J. Math.} \volno{206} (2002), 39--68.

\bibitem{B2} W. Bertram,
\article{On some causal and conformal groups}
\textit{J. Lie Theory} \volno{6} (1996), 215--247.

\bibitem{B3} W. Bertram,
\article{Algebraic structures of Makarevi\v c spaces. I}
\journal{Transform. Groups} \volno{3} (1998), 3--32.

\bibitem{C}
C. Carath{\'e}odory,
\article{The most general transformations of plane regions which transform circles into circles}
\journal{Bull. Amer. Math. Soc.} \volno{43} (1937), 573--579.

\bibitem{F}
J. Faraut,
\article{Int\'egrales de Riesz sur un espace sym\'etrique ordonn\'e}
in: \bookname{Geometry and analysis on finite- and infinite-dimensional Lie groups}
(B\k{e}dlewo, 2000),
Banach Center Publ., Vol. 55.
Polish Acad. Sci., Warsaw, 2002, pages 289--308.

\bibitem{FHO}
J. Faraut, J. Hilgert and G. \'Olafsson,
\article{Spherical functions on ordered symmetric spaces}
\journal{Ann. Inst. Fourier} \volno{44} (1994), 927--966.

\bibitem{HO}
J. Hilgert and G. \'Olafsson,
\bookname{Causal Symmetric Spaces, Geometry and Harmonic Analysis}
Perspectives in Mathematics, Vol.\ 18.
Academic Press, San Diego, 1997.

\bibitem{H}
R. H\"ofer,
\article{A characterization of M\"obius transformations}
\journal{Proc. Amer. Math. Soc.} \volno{128} (2000), 1197--1201.

\bibitem{L}
J. A. Lester,
\article{Conformal spaces}
\journal{J. Geom.} \volno{14} (1980), 108--117.

\bibitem{Z}
E. C. Zeeman,
\article{Causality implies the Lorentz group}
\journal{J. Mathematical Phys.} \volno{5} (1964), 490--493.


\end{thebibliography}
\end{document}